\documentclass[a4paper,12pt]{amsart}
    \usepackage{amsmath,mathrsfs,amssymb,amsthm,amscd}
    \usepackage{a4wide}

    \usepackage{pdfsync}

    \usepackage[]{fontenc}
    \usepackage[all]{xy} 
    \usepackage{hyperref}
     \usepackage{pdfsync}
    \usepackage{yfonts}
    \usepackage{epstopdf}
    \usepackage[pdftex]{graphicx,color}
    \usepackage{pdfpages}
    \usepackage{color}

    \DeclareMathAlphabet{\mathpzc}{OT1}{pzc}{m}{it}
    \newcounter{EQNR}[NNN]
    \DeclareMathOperator{\disc}{disc}

    \DeclareMathOperator{\logc}{log-cond}
    \DeclareMathOperator{\supp}{supp}
    
    \DeclareMathOperator{\rad}{rad}

\newtheorem{thm}{Theorem}[section]
\newtheorem{prop}[thm]{Proposition}
\newtheorem{lemma}[thm]{Lemma}
\newtheorem{cor}[thm]{Corollary}
\newtheorem{conj}[thm]{Conjecture}
\theoremstyle{remark}
\newtheorem{rk}[thm]{Remark}

\renewenvironment{proof}{\par\pagebreak[2]\noindent{\it Proof: }}{
 \hfill $\Box$ \medskip}

\newcommand\Q{\mathbb{Q}}

\newcommand\F{\mathbb{F}}

\newcommand\Z{\mathbb{Z}}

\newcommand\R{\mathbb{R}}

\newcommand\X{\mathcal{X}}
\newcommand\Y{\mathcal{Y}}

\newcommand\Spec{\mathop{\rm Spec}\nolimits}

\renewcommand\O{\mathcal{O}}

\newcommand{\Div}{\operatorname{Div}}

\newcommand{\D}{\mathcal{D}}
\newcommand{\cP}{\mathcal{P}}

\begin{document}
\title[A height inequality implied by abc]{A height inequality for rational points on
elliptic curves implied by the abc-conjecture}

\author{Ulf K\"uhn, J. Steffen M\"uller}
\address{Fachbereich Mathematik\\Universit\"at Hamburg\\
  Bundesstrasse 55\\D-20146 Hamburg}
\email{kuehn@math.uni-hamburg.de}
\email{jan.steffen.mueller@math.uni-hamburg.de}

\date{ \today} 

\maketitle

\begin{abstract}
In this short note we show that the uniform $abc$-conjecture over number fields puts strong restrictions on the
 coordinates of rational points on elliptic curves.
For the proof we use a variant of the uniform $abc$-conjecture over number fields formulated by Mochizuki.
As an application, we generalize a result of Silverman on elliptic non-Wieferich primes.
\end{abstract}
\section{Introduction}
If $E/\Q$ is an elliptic curve in Weierstra\ss{} form and $P\in E(\Q)\setminus\{O\}$, where $O$ is the point
at infinity, then it is well known that we can write 
\[
P=\left(\frac{a_P}{d_P^2},\frac{b_P}{d_P^3}\right),
\] 
where $a_P,\,b_P,\,d_P\in\Z$ satisfy $\gcd(d_P,a_Pb_P)=1$.

The structure of the denominators $d_P$ has been studied, for instance, by
Everest-Reynolds-Stevens \cite{everest-reynolds-stevens}, and has recently received
increasing
attention in the context of elliptic divisibility sequences first studied by Ward
\cite{ward}, see for instance \cite{einsiedler-everest-ward:primes} or \cite{reynolds:perfect} 
and the references therein.
In this paper we make the following conjecture, where $\rad(n)$ denotes the product of
distinct prime divisors of an integer $n$.
\begin{conj}\label{conj:denom}
Let $E/\Q$ be an elliptic curve in Weierstra\ss{} form.
For all $\epsilon>0$ there exists a constant $c_\epsilon$ such that 
\[
\max\left\{\frac{1}{2}\log |a_P|,\log |d_P|\right\}\le(1+\epsilon) \log \rad (d_P) + c_\epsilon
\]
for all $P\in E(\Q)\setminus\{O\}$.
\end{conj}
\begin{rk}\label{rk:siegel}
A strong form of Siegel's Theorem implies the weaker inequality
\[
 \frac{1}{2}\log|a_P| \le (1+\epsilon)\log|d_P|+\O(1),
\]
see for example \cite[Example~IX.3.3]{silverman:elliptic1}. 
\end{rk}
Our main theorem relates Conjecture~\ref{conj:denom}
to the uniform $abc$-conjecture over number fields.

\begin{thm}\label{thm:abc}
Conjecture~\ref{conj:denom} follows from the uniform $abc$-conjecture over number fields.
\end{thm}
\begin{rk}
It is straightforward to generalize both Conjecture~\ref{conj:denom} and
Theorem~\ref{thm:abc} to arbitrary number fields.
For ease of notation we restrict to the rational case here.
\end{rk}
\begin{rk}
Mochizuki \cite{mochizuki:iut4} has recently announced a proof of the uniform $abc$-conjecture over
number fields.
\end{rk}

We now list some consequences of Conjecture~\ref{conj:denom}.
\begin{prop}\label{prop:sqfree}
Suppose that Conjecture~\ref{conj:denom} holds and let $E/\Q$ be an elliptic curve in 
Weierstra\ss{} form.
Then the set of all $P\in E(\Q)\setminus\{O\}$ such that the squarefree part of $d_P$ is bounded is
finite.
\end{prop}
If the bound on the squarefree part of $d_P$ in Proposition~\ref{conj:denom} is~1, then we
get a conditional proof of Siegel's Theorem that there are only finitely many integral
points on $E$, and we can also deduce:
\begin{cor}\label{cor:powers}
Suppose that Conjecture~\ref{conj:denom} holds and let $E/\Q$ be an elliptic curve in 
Weierstra\ss{} form.
Then the set of all $P\in E(\Q)\setminus\{O\}$ such that $d_P$ is a perfect power 
is finite.
\end{cor}
\begin{rk}
It is shown in \cite[Theorem~1.1]{everest-reynolds-stevens} that for a fixed exponent $n>1$, there are
only finitely many $P\in E(\Q)\setminus\{O\}$ such that $d_P$ is an $n$th power.
According to \cite[Remark~1.2]{everest-reynolds-stevens}, the uniform $abc$-conjecture over 
number fields  implies that for $n\gg0$, there are no $P\in E(\Q)\setminus\{O\}$ such that
$d_P$ is an $n$th power. 
Together, these results also imply
that the finiteness of the set of $P\in E(\Q)\setminus\{O\}$ such that $d_P$ is a perfect power
is a consequence of the uniform $abc$-conjecture over number fields.
However, a direct proof of the assertion from \cite[Remark~1.2]{everest-reynolds-stevens} 
has not been published and, according to Reynolds \cite{reynolds:personal}, is rather complicated.
\end{rk}

Another application of Conjecture~\ref{conj:denom} concerns {\em elliptic non-Wieferich primes}.
For a prime $p$, we define $N_p:=\#E(\F_p)$. 
If $P\in E(\Q)$ is non-torsion, let
\[
    W_{E,P}:=\left\{p\;\;\mathrm{prime}:N_pP\not\equiv O\pmod{p^2}\right\}
\]
be the set of elliptic non-Wieferich primes to base $P$.
The following result is due to Silverman. 
\begin{thm}\label{thm:silverman}(Silverman, \cite[Theorem~2]{silverman:wieferich})
Assume that the $abc$-conjecture (over $\Q$) holds.
If an elliptic curve $E/\Q$ has $j$-invariant equal to~0 or~1728 and if $P\in E(\Q)$ is non-torsion, then
\begin{equation}\label{eq:wieferich}
    \left|\{p\in W_{E,P}:p\le X\}\right|\ge \sqrt{\log(X)}+\O_{E,P}(1)\qquad\mathrm{as}\quad X\to\infty.
\end{equation}
\end{thm}
This is the analogue of \cite[Theorem~1]{silverman:wieferich}, giving an asymptotic lower bound
(dependent on the $abc$-conjecture over $\Q$) for the number of classical non-Wieferich primes up to
a given bound.
In particular, this proves that the $abc$-conjecture over $\Q$ implies the existence of
infinitely many elliptic non-Wieferich primes to any base $P\in E(\Q)$ if
$j(E)\in\{0,\,1728\}$.
See \cite{voloch:elliptic} for further results concerning elliptic non-Wieferich primes.

If we assume Conjecture~\ref{conj:denom} instead of the $abc$-conjecture over $\Q$, we can
eliminate the condition on the $j$-invariant of $E$.
\begin{prop}\label{prop:wieferich}
    Assume Conjecture~\ref{conj:denom} and let $E/\Q$ be an elliptic curve.
Then~\eqref{eq:wieferich} holds for every non-torsion $P\in E(\Q)$.
\end{prop}

In Section~\ref{sec:mochizuki} we recall work of Mochizuki from
\cite{mochizuki:elliptic_general}, which we use in Section~\ref{sec:thm} 
for the proof of Theorem~\ref{thm:abc}. 
We prove Proposition~\ref{prop:sqfree} in Section~\ref{sec:prop} and
Proposition~\ref{prop:wieferich} in Section~\ref{sec:wieferich}.
\section*{Acknowledgements}
We thank Joe Silverman for helpful comments and his suggestion that
Proposition~\ref{prop:wieferich} should hold.
We also thank Jonathan Reynolds for helpful discussions.
The second author was supported by DFG-grant KU 2359/2-1.

\section{Mochizuki's height inequality}\label{sec:mochizuki}
In this section, we define Mochizuki's   {\em log-conductor function}
\cite[\S1]{mochizuki:elliptic_general} and state his Conjecture~\ref{conj:mochizuki}.
We assume some familiarity with the basics of Arakelov theory, see for
instance~\cite{soule:hermitian}.

Let $K$ be a number field, let $X$ be a smooth, proper, geometrically connected
curve over $K$ and let $D$ be an effective divisor on $X$.
Extend $X$ to a proper normal model $\X$ which is flat over $\Spec(\O_K)$ and extend $D$
to an effective horizontal divisor
$\D\in\Div(\X)$.
We can define a function $\logc_{\X,\D}$ on $X(\overline{K})$ as follows:
Let $P\in X(\overline{K})$ and let $F$ be a number field containing $P$.
Then $P$ induces a morphism $\cP:\Spec(\O_F)\to\Y$, where $\Y$ is the normalisation of 
$\X\times\Spec(\O_F)$ and we define 
\[
 \logc_{\X,\D}(P):=
\frac{1}{[F:\Q]}\widehat{\deg_F}\left((\cP^*\D)_{\mathrm{red}}\right)\in\R,
\]
where $\widehat{\deg_F}$ is the arithmetic degree of an arithmetic divisor on
$\Spec(\O_F)$.
\begin{rk}\label{rk:mochizuki_abc}
Note that up to a bounded function, $\logc_{\X,\D}$ only depends on $X$ and $D$ (see
\cite[Remark~1.5.1]{mochizuki:elliptic_general}).
\end{rk}
\begin{rk}
Alternatively, we could define the log-conductor function as follows:
Extend $X$ to a proper regular model $\X$ over $\Spec(\O_K)$ and $D$ to an effective
horizontal divisor $\D\in\Div(\X)$.
Let $\pi:\X'\to\X\times\Spec(\O_F)$ be the minimal desingularization and let
$\overline{P}\in\Div(\X')$ be the Zariski closure of $P$.
Then we define 
\[
\logc'_{\X,\D}(P):=\frac{1}{[F:\Q]}\sum_{\mathfrak{p}\in S} \log Nm(\mathfrak{p})\in\R,
\]
where $S$ is the set of finite primes of $F$ such that the intersection multiplicity
$(\overline{P},\pi^*\D)_\mathfrak{p}\ne 0$.
Then it is easy to see that $\logc_{\X,\D}=\logc'_{\X,\D}+\O(1)$.
\end{rk}

The following variation of Vojta's height conjecture is due to Mochizuki:
\begin{conj}\label{conj:mochizuki}(Mochizuki, \cite[\S2]{mochizuki:elliptic_general})
Let $X$ be a smooth, proper, geometrically connected curve over a number field $K$. 
Let $D\subset X$ be an effective reduced divisor, $U_X:=X\setminus\supp(D)$, $d$ a positive integer and
$\omega_X$ the canonical sheaf on $X$.
Fix a proper normal model $\X$ of $X$ which is flat over $\Spec(\O_K)$  and extend
$D$ to an effective horizontal divisor $\D$ on $\X$.
Suppose that $\omega_X(D)$ is ample and let  $h_{\omega_X(D)}$ be a Weil height function on $X$ 
with respect to $\omega_X(D)$.

If $\epsilon>0$, then there exists a constant $c_{\epsilon,d,\X,\D}$ such that
\[
h_{\omega_X(D)}(P)\le (1+\epsilon)\left(\log\disc(k(P))+\logc_{\X,\D}(P)\right)+
c_{\epsilon,d,\X,\D}
\]
for all $P\in X(\overline{K})$ such that $[k(P):\Q]\le d$, where $k(P)$ is the minimal
field of definition of $P$ (as a point over $\overline{\Q}$).
 \end{conj}
\begin{rk}\label{rk:abc}
Mochizuki \cite[Theorem~2.1]{mochizuki:elliptic_general} proves that
Conjecture~\ref{conj:mochizuki} follows from the uniform $abc$-conjecture over number
fields.
\end{rk}
\section{Proof of Theorem~\ref{thm:abc}}\label{sec:thm}
\begin{proof}
We specialize Conjecture~\ref{conj:mochizuki} to the case $K=\Q$, $X=E$, $d=1$ and
$D=(O)$.
Let $P\in E(\Q)\setminus\{O\}$; then we have
\[
h_{\omega_E(D)}(P)=h_D(P) = \max\left\{\frac{1}{2}\log |a_P|,\log |d_P|\right\} + \O(1).
\]
In order to compute the log-conductor of $P$ we consider the model $\X$ over $\Spec(\Z)$ determined 
by the given Weierstra\ss{} equation of $E$ and
extend $D$ to $\D\in\Div(\X)$ by taking the Zariski closure.
Then a prime number $p$ lies in the support of $\cP^*\D$ if and only if $p\mid d_P$ and hence
we get
\[
    \logc_{\X,\D}(P)=\log\rad(d_P).
\]
Therefore Conjecture~\ref{conj:mochizuki} implies Conjecture~\ref{conj:denom}.
Using Remark~\ref{rk:abc}, this finishes the proof of the theorem. 
\end{proof}
\section{Proof of Proposition~\ref{prop:sqfree}}\label{sec:prop}
\begin{proof}
Let $0<\epsilon \ll 1$ be small and let $c_\epsilon$ be the corresponding constant from
Conjecture~\ref{conj:denom}.
Suppose that $P\in E(\Q)\setminus\{O\}$ satisfies
\[
    d_P=d'_P\cdot\prod^n_{i=1} p_i^{t_i},
\]
where $d'_P$ is squarefree, $p_1,\ldots,p_n$ are primes and $t_1,\ldots,t_n$ are integers 
such that $t_i>1$ for all $i$.
Then, according to Conjecture~\ref{conj:denom}, we must have
\[
    \sum^n_{i=1} (t_i-1-\epsilon)\log p_i \le c_\epsilon+\epsilon\log d'_P.
\]
This implies that if $d'_P$ is bounded from above, 
then $\sum^n_{i=1} (t_i-1-\epsilon)\log p_i$ is bounded from above as well.
Hence $d_P$ is bounded from above, as is the height of $P$ by Remark~\ref{rk:siegel}.
This proves the corollary, as there are only finitely many $P$ of bounded height.
\end{proof}  

\section{Proof of Proposition~\ref{prop:wieferich}}\label{sec:wieferich}
Let $E/\Q$ be an elliptic curve in Weierstra\ss{} form and let $P\in E(\Q)$ have infinite order.
Note that the only place in Silverman's proof of Theorem~\ref{thm:silverman} where the
assumption $j(E)\in\{0,\,1728\}$ is invoked is in the proof of
\cite[Lemma~13]{silverman:wieferich}.
For $Q\in E(\Q)\setminus\{O\}$ we write $d_{Q}=d'_{Q}\cdot v_{Q}$, where $d'_Q$ is as in the proof of
Proposition~\ref{prop:sqfree}).

In order to deduce the statement of \cite[Lemma~13]{silverman:wieferich}, it suffices to show that for all
$\epsilon>0$ there exists a constant $c'_{\epsilon}$ such that
\[
\log v_{nP}\le\epsilon \log(d_{nP})+c'_{\epsilon}
\]
for all $n\ge 1$.
If we assume Conjecture~\ref{conj:denom}, then we can in fact prove a stronger result:
\begin{lemma}
Assume Conjecture~\ref{conj:denom}.
Then for all $\epsilon>0$ there exists a constant $c'_{\epsilon}$ such that
\[
\log(v_Q)\le \epsilon\log\rad(d_Q)+c'_\epsilon.
\]
for all $Q\in E(\Q)\setminus\{O\}$.
\end{lemma}
\begin{proof}
Let $\epsilon>0$, let $\epsilon'=2\epsilon$ and let $c_{\epsilon'}$ be the constant from Conjecture~\ref{conj:denom}.
Since $d'_Q=\rad(d'_Q)$, Conjecture~\ref{conj:denom} predicts
\begin{equation}\label{eq:vq}
 \log \left(\frac{v_Q}{\rad(v_Q)}\right)\le \epsilon'\log\rad(d_Q)+c_{\epsilon'}
\end{equation}
for any $Q\in E(\Q)\setminus\{O\}$.
But since, by construction, $v_Q$ is not exactly divisible by any prime number, we also get
\[
\log\rad(v_Q)\le \log \left(\frac{v_Q}{\rad(v_Q)}\right).
\]
The latter is at most $\epsilon'\log\rad(d_Q)+c_{\epsilon'}$ by~\eqref{eq:vq}.
Rewriting~\eqref{eq:vq}, we conclude
\[
 \log {v_Q}\le
2\epsilon'\log\rad(d_Q)+2c_{\epsilon'}=\epsilon\log\rad(d_Q)+2c_{\epsilon'}.
\]
\end{proof}

\bibliography{mochizuki-biblio}
\bibliographystyle{alpha}

\end{document}